\documentclass[11pt]{amsart}
\usepackage[margin=1.0in]{geometry}
\usepackage{bm, enumerate, amssymb, ulem}
\usepackage{color, graphicx, hyperref}

\newcommand{\R}{\mathbb{R}}

\newcommand{\C}{\mathbb{C}}
\newcommand{\D}{\mathbb{D}}
\renewcommand{\Re}{\operatorname{Re}}
\renewcommand{\Im}{\operatorname{Im}}

\numberwithin{equation}{section}
\numberwithin{figure}{section}

\theoremstyle{plain} 
\newtheorem{theorem}{Theorem}[section]
\newtheorem{lemma}[theorem]{Lemma}

\newtheorem{proposition}[theorem]{Proposition}

\newtheorem{fact}[theorem]{Fact}
\theoremstyle{definition} 

\newtheorem{remark}[theorem]{Remark}

\title[Doubly periodic minimal surfaces]
{Genus three embedded doubly periodic minimal surfaces with parallel ends}

\date{April 15, 2026}

\author[P.~Connor]
{Peter Connor}
\address[Peter Connor]{Department of Mathematical Sciences\\Indiana University South Bend\\South
Bend\\IN 46634\\USA}
\email{pconnor@iusb.edu}
\thanks{The first author was supported in part by a Faculty Research Grant from Indiana University South Bend.}

\author[S.~Fujimori]
{Shoichi Fujimori}
\address[Shoichi Fujimori]{Department of Mathematics\\Hiroshima University\\Higashihiroshima\\Hiroshima 739-8526\\Japan}
\email{fujimori@hiroshima-u.ac.jp}
\thanks{
The second author was supported in part by Grant-in-Aid for Scientific Research (C) 
No. 21K03226 and 25K06977 from Japan Society for the Promotion of Science.
}

\author[P.~Marmorino]
{Phillip Marmorino}
\address[Phillip Marmorino]{Department of Earth, Atmospheric, and Planetary Science\\Purdue University\\West Lafayette\\IN  47907\\USA}
\email{pmarmori@purdue.edu}
\thanks{The third author was supported in part by a SMART Summer Fellowship from Indiana University South Bend.}

\author[T.~Shoda]
{Toshihiro Shoda}
\address[Toshihiro Shoda]{Department of Mathematics\\Kansai University\\Suita\\Osaka 564-8680\\Japan}
\email{tshoda@kansai-u.ac.jp}
\thanks{
The fourth author was supported in part by Grant-in-Aid for Scientific Research (C) 
No. 20K03616 and 24K06750 from Japan Society for the Promotion of Science.
}

\makeatletter
\@namedef{subjclassname@2020}{%
  \textup{2020} Mathematics Subject Classification}
\makeatother


\begin{document}

\begin{abstract}
We construct a one-parameter family of embedded doubly periodic minimal surfaces of genus three with four parallel ends.
The Weierstrass data for each surface of the family are given and the two dimensional period problem is solved.
\end{abstract}

\maketitle

\section{Introduction}
A minimal surface $M$ in Euclidean space is called {\it doubly periodic} if it is invariant under two linearly independent translations, which can be assumed to be horizontal.  The first example of a doubly periodic minimal surface was discovered by Scherk \cite{sche1}.

Let $L$ be the two-dimensional lattice generated by the maximal group of these translations.  Meeks and Rosenberg \cite{mr3} showed that if $M/L$ is complete, properly embedded, and of finite topology then the quotient has a finite number of annular top and bottom ends, each of which is asymptotic to flat annuli.  This type of end is referred to as a {\it Scherk end}.  The top and bottom ends are either parallel or not.  In the parallel case, Meeks and Rosenberg proved that the number of top and bottom ends are equal to the same even number.  

Karcher \cite{ka4} and Meeks-Rosenberg \cite{mr3} constructed a three-parameter family of genus one examples with parallel ends.  Wei \cite{wei2} added a handle to one of the Karcher and Meeks-Roseberg examples to produce a one-parameter family of genus two examples.  Connor and Weber \cite{cw1} proved that, for each genus $g$, there is a three-parameter family of genus $g$ examples with parallel ends that have a foliation of $\R^3$ by vertical planes as a limit.  The Connor-Weber examples are only known to exist in a small neighborhood of the limit foliation.  The Karcher and Meeks-Rosenberg genus one examples and the Wei genus two example are all members of the Connor-Weber examples.  The long-term deformations of the Karcher and Meeks-Rosenberg examples and the Wei examples are already known.  Previously, the long-term deformation of higher genus Connor-Weber examples had not been established.  We address this issue in this paper in the case of one of the genus three Connor-Weber examples.  This is an important step in studying the moduli space of the Connor-Weber examples.  It also provides the opportunity to consider other possible limits of these families of surfaces.

The goal of this paper is to construct a one-parameter family $M_{\lambda}$ of embedded doubly periodic minimal surfaces with parallel ends with genus three in the quotient.  The parameter $\lambda$ ranges over a certain  open interval $(a_1,a_2)$ in $(0,1)$.  See Figure \ref{fig1}.  These surfaces are constructed using the Weierstrass Representation by adding a handle to Wei's genus two surfaces along a horizontal symmetry plane.

\begin{theorem}
There exists a one-parameter family $M_{\lambda}$ of properly embedded doubly periodic minimal surfaces with parallel ends such that $M_{\lambda}/L$ has genus three and four parallel ends, where $L$ is the full period lattice of the surface.
\label{thm1}
\end{theorem}

This theorem is proven in Sections \ref{wr} through \ref{embedded}. 
The data of the surface $M_{\lambda}$ are given in Section \ref{wr}, Equations \eqref{eq:wr}, \eqref{eq:R}, and \eqref{eq:G-dh} with three parameters $\lambda$, $\lambda_1$, and $\lambda_2$. 
Then in Section \ref{period}, Proposition~\ref{pr:existence}, we prove that for any $\lambda\in (a_1,a_2)$ there exist $\lambda_1$ and $\lambda_2$ so that the surface $M_\lambda$ is well-defined as a doubly periodic minimal surface. 
The embeddedness of $M_\lambda$ is proven in Section \ref{embedded}, Proposition~\ref{pr:embedded}. 
Although we show the existence of $M_\lambda$ for $\lambda\in (a_1,a_2)\subset (0,1)$, 
our numerical experiments show that $M_\lambda$ exists for any $\lambda\in (0,1)$.  
Proposition~\ref{pr:existence} is proven by utilizing the Poincare-Miranda theorem. %
The restriction of the range of $\lambda$ is needed in order to ensure the period problem has a solution.

\begin{figure}[htbp] 
\centering
 \includegraphics[height=2in]{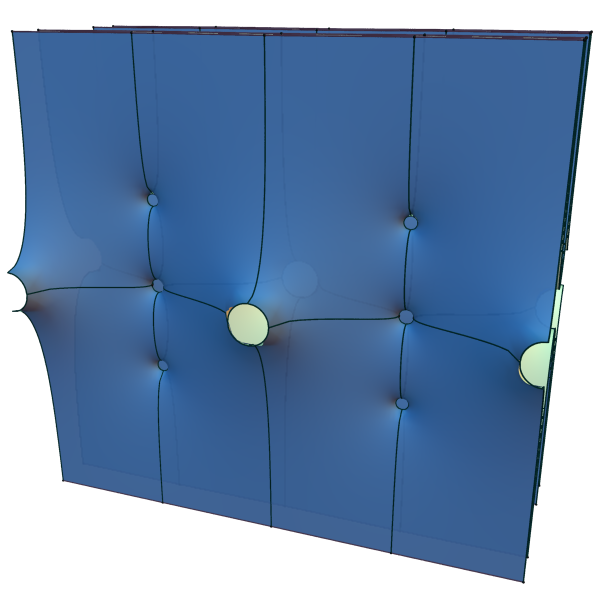} \
 \includegraphics[height=2in]{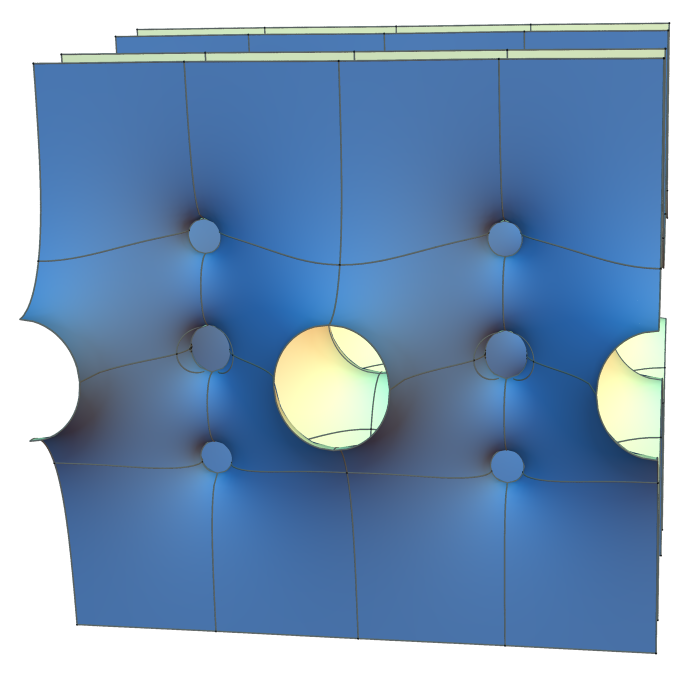} \
 \includegraphics[height=2in]{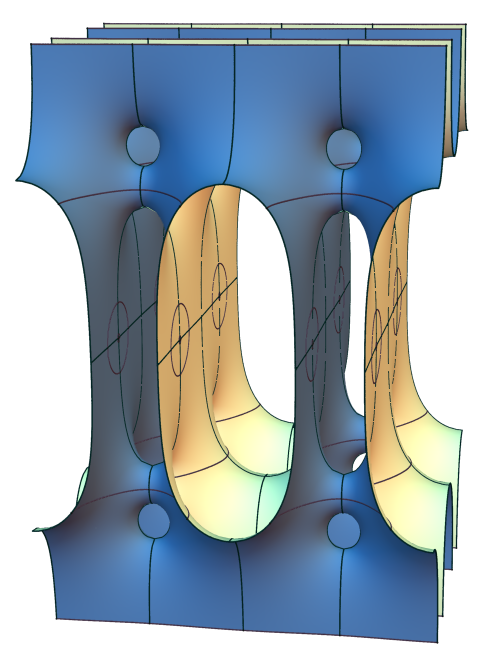} 
\caption{$M_{\lambda}$ with $\lambda=0.2$ (left), $\lambda=0.5$ (center), and $\lambda=0.99$ (right).}
\label{fig1}
\end{figure} 

Connor and Weber showed the existence of these surfaces in a small neighborhood of their limit as a foliation of $\R^3$ by vertical planes.  This result significantly expands the known parameters for which these surfaces exist.  As $\lambda\rightarrow0$, $M_{\lambda}$ limits as a foliation of vertical planes, the Connor-Weber limit.  Numerical experiments indicate that, as $\lambda\rightarrow1$, $M_{\lambda}$ limits as two copies of a genus one Scherk surface glued together end-to-end.

Weber \cite{web} constructed a two-parameter family of triply periodic minimal surfaces of genus 5, which is intuitively considered as the result of adding two handles to a family of Schwarz CLP surfaces.  
See Figure~\ref{fg:tpms}. 
The surfaces $M_{\lambda}$ can be considered as limiting surfaces of these triply periodic minimal surfaces. 

\begin{figure}[htbp] 
\centering
 \includegraphics[height=2in]{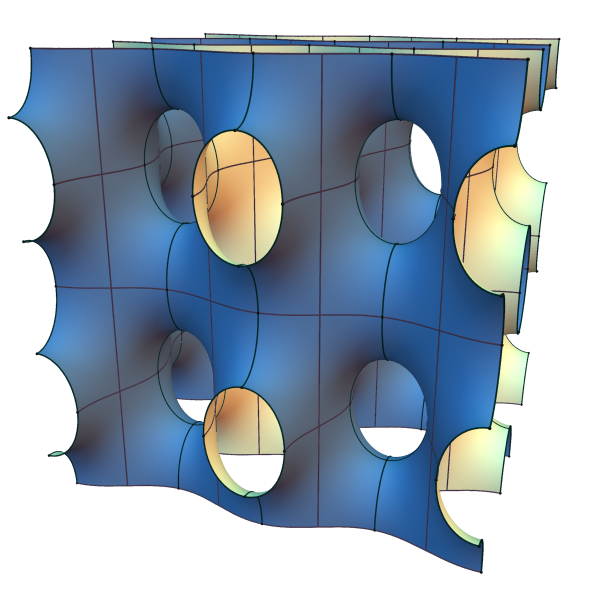} \
 \includegraphics[height=2in]{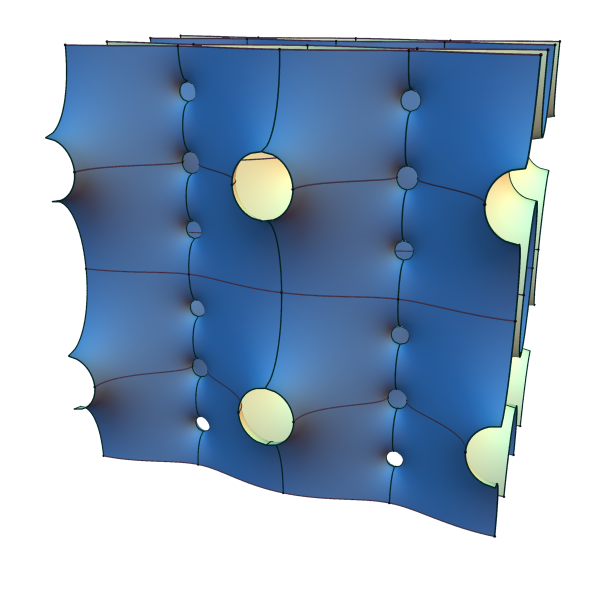} \
 \includegraphics[height=2in]{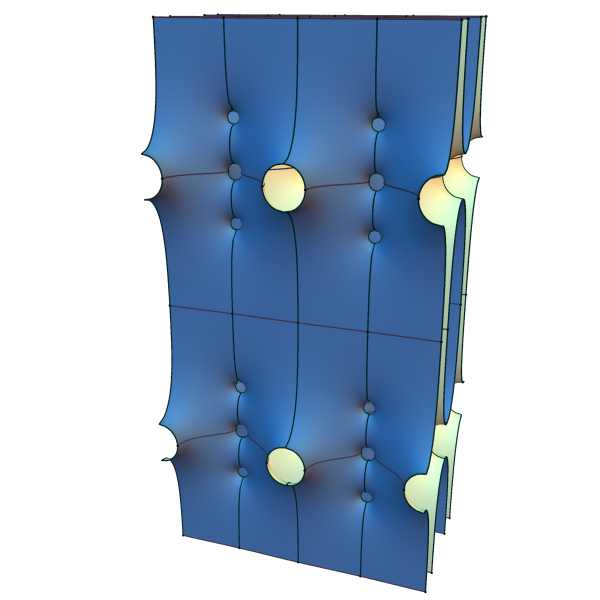} 
\caption{Left: Triply periodic minimal surface of genus 3, called Schwarz CLP. Center: Weber's triply periodic minimal surface of genus 5. Right: A limit of the Weber's surface.}
\label{fg:tpms}
\end{figure} 

 
\section{Weierstrass Representation}
\label{wr}
We use the Weierstrass Representation of minimal surfaces. 
Let $M$ be a minimal surface and $R$ the underlying Riemann surface of $M$.  
Then $M$ can be expressed by $X:R\to\R^3$, 
\begin{equation}\label{eq:wr}
X(z)=\Re\int_{z_0}^z(\phi_1,\phi_2,\phi_3),
\end{equation}
where 
\[
(\phi_1,\phi_2,\phi_3)=\left(\frac{1}{2}\left(\frac{1}{G}-G\right)dh,\frac{i}{2}\left(\frac{1}{G}+G\right)dh,dh\right),
\]
$z_0,z\in R$, $G$ a meromorphic function called the Gauss map, and $dh$ a holomorphic one-form called the height differential.  
This is the Weierstrass Representation.  See for example \cite{ka6} for details. 
We call $(R,G,dh)$ the Weierstrass data of the minimal surface $M$.

Conversely, given a Riemann surface $R$ which is biholomorphic to a compact Riemann surface $\overline{R}$ with finitely many points removed, a meromorphic function $G$ and a holomorphic one-form $dh$ on $R$, the triple $(R,G,dh)$ is the Weierstrass representation for a doubly periodic minimal surface with horizontal periods $\bm{v}_1$ and $\bm{v}_2$ and vertical Scherk ends if the following conditions are met.
\begin{enumerate}
\item\label{item:zeros}
The zeros of $dh$ are the zeros and poles of $G$ on $R$ with the same multiplicity.
\item\label{item:ends}
$G$ and $dh$ extend meromorphically to $\overline{R}$, and 
the height differential $dh$ has poles of order one and $G$ has finite value at the ends $\overline{R}-R$.
\item\label{item:period}
For each closed curve $\gamma$ on $R$,
\[
\Re\int_{\gamma}(\phi_1,\phi_2,\phi_3)=(0,0,0)\pmod{\bm{v}_1,\bm{v}_2}.
\]
\end{enumerate}
 
Of course, $X$ depends on the path of integration.  Condition \ref{item:period} ensures that $X$ is well-defined and is called the {\it period problem}.

\begin{remark}\label{rm:conj}
Denoting the universal cover of $R$ by $\widetilde{R}$, the minimal surface $X^*:\widetilde{M}\to\R^3$ with the Weierstrass data $(R,G,idh)$ is called the {\it conjugate surface} to $M$, and is denoted by $M^*$. 
It is known that any curve of $R$ which is mapped by $X$ to a nonstraight planar geodesic of $M$ is mapped by $X^*$ to a straight line in $M^*$. 
Furthermore, since the Gauss map $G$ and the first fundamental form $(|G|^{-1}+|G|)^2|dh|^2/4$ are the same for both $M$ and $M^*$, it follows that the planar geodesic in $M$ will lie in a plane perpendicular to the corresponding line in $M^*$ and that the planar geodesic in $M$ will have the same length as the line in $M^*$.
\end{remark}

\begin{figure}[htbp] 
\centering
 \includegraphics[width=2.2in]{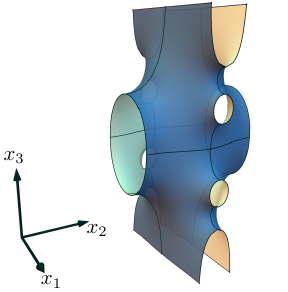} \ \ 
 \includegraphics[width=2.2in]{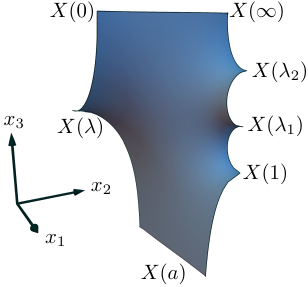} 
\caption{The constructed surface (left) and a fundamental domain, one-eighth of the surface (right).}
\label{fp-fp8}
\end{figure} 

To give the Weierstrass data $(R,G,dh)$ of our minimal surfaces, we need some preparations.  Throughout this paper, we assume $\lambda$, $\lambda_1$, and $\lambda_2$ 
are real constants such that 
$$0<\lambda<1<\lambda_1<\lambda_2.$$   

The surfaces in this paper are constructed by adding a handle to Wei's genus two doubly periodic minimal surfaces \cite{wei2}. 
We would like to keep the same symmetry, embeddedness, and number of ends as Wei's surfaces. 
See Figure~\ref{fp-fp8}. 
The underlying Riemann surface $R$ is hyperelliptic with 8 branch points minus 4 points. 
Since each surface is symmetric with respect to 3 coordinate planes, the surface consists of 8 congruent pieces, and we may assume that the first quadrant of the $z$-plane of $R$ corresponds to one of the 8 pieces.  
Corresponding to the saddle points at the bottom and top of each handle on the surface, the Gauss map $G(z)$ needs zeros at $z=\lambda$ and $z=\lambda_1$ and poles at $z=1$ and $z=\lambda_2$. 
Furthermore, we want $G(z)\in\mathbb{R}$ for $z\in [0,\lambda)\cup (1,\lambda_1)\cup (\lambda_2,\infty)$ and 
$G(z)\in i\mathbb{R}$ for $z\in (\lambda,1)\cup (\lambda_1,\lambda_2)$, 
and $|G(z)|=1$ for $z\in i\mathbb{R}_{\ge 0}$. 
Then due to the last condition above, the image of the imaginary axis lies in a plane parallel to the $(x_1,x_2)$-plane. 
Taking a reflection with respect to this plane, we have that $G(z)$ has poles at $z=-\lambda$ and $z=-\lambda_1$ and zeros at $z=-1$ and $z=-\lambda_2$. 
We set 
\begin{align*}
g_1(z)&=(z-\lambda)(z+1)(z-\lambda_1)(z+\lambda_2), \\
g_2(z)&=(z+\lambda)(z-1)(z+\lambda_1)(z-\lambda_2). 
\end{align*}Then $G(z)$ must satisfy 
\[
G(z)^2=\frac{g_1(z)}{g_2(z)}. 
\]
Next we want the surface to have an end at $z=a$ which is asymptotic to vertical planes parallel to the $(x_1,x_3)$-plane.  
This happens when $G(a)^2=-1$, which takes place when 
\begin{equation}\label{eq:a}
a=\sqrt{\dfrac{\alpha -\sqrt{\alpha^2-4\lambda\lambda_1\lambda_2}}{2}},
\end{equation}
where 
\begin{equation}\label{eq:alpha}
\alpha=\lambda+\lambda_1-\lambda_2-\lambda\lambda_1+\lambda\lambda_2+\lambda_1\lambda_2.
\end{equation}
Then the height differential $dh$ must be written as 
$
dh=dz/(z^2-a^2).
$

Now we consider the properties of the parameters of the surfaces. 
First, the above $g_1(z)$ and $g_2(z)$ satisfy  
\[
g_1(z)+g_2(z)=2(z^2-a^2)(z^2-b^2), 
\]
where 
\begin{equation}\label{eq:b}
  b=\sqrt{\dfrac{\alpha +\sqrt{\alpha^2-4\lambda\lambda_1\lambda_2}}{2}},
\end{equation}
and $a$ and $\alpha$ are as in \eqref{eq:a} and \eqref{eq:alpha} respectively.
It is easy to verify that $\alpha$ satisfies 
\[
2\lambda < \alpha < 2\lambda_2^2. 
\] 
Also, $a$ and $b$ satisfy the following lemma. 
\begin{lemma}\label{lm:ab}
Both $a$ and $b$ as above are real and satisfy
\[
\lambda < a < 1
\qquad\text{and}\qquad
\lambda_1 < b < \lambda_2.
\]
\end{lemma}

\begin{proof}
By direct computations, we have
\begin{align}
(\alpha^2-4\lambda\lambda_1\lambda_2)-(\alpha-2\lambda)^2 
&=-4\lambda (1-\lambda)(\lambda_2-\lambda_1)<0, 
\label{eq:a-2l} \\
(\alpha^2-4\lambda\lambda_1\lambda_2)-(\alpha-2)^2 
&=4(1-\lambda)(\lambda_1-1)(1+\lambda_2)>0, 
\label{eq:a-2} \\
(\alpha^2-4\lambda\lambda_1\lambda_2)-(\alpha-2\lambda_1^2)^2 
&=4\lambda_1(\lambda_1-1)(\lambda+\lambda_1)(\lambda_2-\lambda_1)>0, 
\label{eq:a-2l1} \\
(\alpha^2-4\lambda\lambda_1\lambda_2)-(\alpha-2\lambda_2^2)^2 
&=-4\lambda_2(\lambda_2-\lambda)(1+\lambda_2)(\lambda_2-\lambda_1)<0.
\label{eq:a-2l2} 
\end{align}
By \eqref{eq:a-2} or \eqref{eq:a-2l1}, 
we see $\alpha^2-4\lambda\lambda_1\lambda_2>0$ 
and hence both $a$ and $b$ are positive real. 
Since $\alpha >2\lambda$, \eqref{eq:a-2l} yields 
\[
\sqrt{\alpha^2-4\lambda\lambda_1\lambda_2} < \alpha-2\lambda . 
\]
This shows $\sqrt{\lambda}<a$. 
By \eqref{eq:a-2}, we have
\[
\sqrt{\alpha^2-4\lambda\lambda_1\lambda_2}>\big|\alpha-2\big|.
\]
Thus
\[
a^2=\frac{\alpha-\sqrt{\alpha^2-4\lambda\lambda_1\lambda_2}}{2}
<\frac{\alpha-\big|\alpha-2\big|}{2}
=\begin{cases}
 1 & (\alpha-2\ge 0) \\
 \alpha -1 & (\alpha -2<0). \\
 \end{cases}
\]
This implies $a<1$.
Similarly, \eqref{eq:a-2l1} and \eqref{eq:a-2l2} imply $\lambda_1<b$ and $b<\lambda_2$, respectively. 

Therefore we have
\begin{equation}\label{eq:const}
0<\lambda<\sqrt{\lambda}<a<1<\lambda_1<b<\lambda_2. 
\qedhere
\end{equation}
\end{proof}

Now we are ready to give the Weierstrass data of our surfaces.  
We consider the following compact Riemann surface of genus 3
\[
\overline{R}=\{(z,w)\in (\mathbb{C}\cup\{\infty\})^2\;;\;g_2(z)w^2=g_1(z)\}
\]
and define
\begin{equation}\label{eq:R}
R=\overline{R}\setminus\{(a,i),(-a,i),(a,-i),(-a,-i)\}. 
\end{equation}
Define the Gauss map and height differential by
\begin{equation}\label{eq:G-dh}
G=w,\qquad dh=\frac{dz}{z^2-a^2}. 
\end{equation}
We denote by $M_\lambda$ the minimal surfaces with the Weierstrass data \eqref{eq:R} and \eqref{eq:G-dh}.  Note that Weierstrass Representation conditions \ref{item:zeros} and \ref{item:ends} are satisfied by $G$ and $dh$. In Section \ref{period} we prove that for each $\lambda\in (0,1)$, there exist $\lambda_1$ and $\lambda_2$ so that the period problems are solved. 
This is why we denote the subscript $\lambda$ of $M_\lambda$. 


The following lemma is obvious. 

\begin{lemma}\label{lm:sym}
The Riemann surface $\overline{R}$ has automorphisms 
\[
\tau_1:(z,w)\mapsto(\overline{z},-\overline{w}),\qquad
\tau_2:(z,w)\mapsto(\overline{z},\overline{w}),\qquad
\tau_3:(z,w)\mapsto(-\overline{z},\overline{w}^{-1}),
\]
with 
\[
\tau_1^*(\phi_1,\phi_2,\phi_3)=(-\overline{\phi_1},\overline{\phi_2},\overline{\phi_3}),\quad
\tau_2^*(\phi_1,\phi_2,\phi_3)=(\overline{\phi_1},-\overline{\phi_2},\overline{\phi_3}),\quad
\tau_3^*(\phi_1,\phi_2,\phi_3)=(\overline{\phi_1},\overline{\phi_2},-\overline{\phi_3}).
\]
\end{lemma}

\section{Period Problem}
\label{period}

In this section we prove the following proposition.

\begin{proposition}\label{pr:existence}
We may choose an open interval $(a_1,a_2)\subset (0,1)$ satisfying the following 
property{\rm :}
for each $\lambda\in (a_1,a_2)$, 
there exist $\lambda_1$, $\lambda_2$ $(1<\lambda_1<\lambda_2)$ and a rectangular torus $\mathbb{T}^2$ such that \eqref{eq:wr}, considered as a map into $\mathbb{T}^2\times\mathbb{R}$,  with the Weierstrass data \eqref{eq:R} and \eqref{eq:G-dh} is single-valued on $R$. 
\end{proposition}

Since $\phi_3$ can be written as
\[
\phi_3=d\left(\frac{1}{2a}\log\frac{z-a}{z+a}\right), 
\]
we see that 
\[
\Re\int_\gamma\phi_3 = 0
\]
for any loop $\gamma$ on $R$. 
 
Now we consider the periods of $\phi_1$ and $\phi_2$ for loops around the ends. 
Note that $\phi_2$ is holomorphic on $\overline{R}$.  In fact, we have
\[
\phi_2=\frac{i}{2}\frac{g_1(z)+g_2(z)}{g_2(z)w(z^2-a^2)}dz=\frac{i(z^2-b^2)}{g_2(z)w}dz, 
\]
and hence $\phi_2$ has no residues at the ends.  That is, 
$\phi_2$ has no periods at ends. 
We also note that the residues of $\phi_1$ at ends are either $i/(2a)$ or $-i/(2a)$. 
Therefore, for any loop $\gamma$ around the ends, we have
\[
\Re\int_{\gamma}(\phi_1,\phi_2,\phi_3)\in \{n\bm{v}_1:n\in\mathbb{Z}\},
\qquad
\bm{v}_1=(\pi/a,0,0).
\]

Next we consider the periods of $\phi_1$ and $\phi_2$ for the homology basis of $\overline{R}$. 
We fix the $z$-plane in $\overline{R}$ with $w(\infty)=+1$. 
Let $\gamma_j$ ($j=1,2,3$) be loops in the $z$-plane winding once 
around $[\lambda,1]$, $[1,\lambda_1]$, $[\lambda_1,\lambda_2]$, respectively, 
counterclockwise, 
and we denote by the same notations the lifts of $\gamma_j$ to $\overline{R}$. 
Then $\{\gamma_j,\tau_3\circ\gamma_j:j=1,2,3\}$ gives a homology basis of $\overline{R}$, where $\tau_3$ is defined in Lemma \ref{lm:sym}. 
By Lemma \ref{lm:sym}, it suffices consider the following six periods
\[
\Re\int_{\gamma_j}\phi_k,\qquad
(j=1,2,3,\;\;k=1,2).
\]
By the symmetry arguments as in \cite[Proof of Lemma 3.2]{wei2}, we have 
\[
\Re\int_{\gamma_1}\phi_1=
\Re\int_{\gamma_2}\phi_2=
\Re\int_{\gamma_3}\phi_1=0.
\]
So all we need to consider are the periods of 
\[
\Re\int_{\gamma_1}\phi_2,\qquad
\Re\int_{\gamma_2}\phi_1,\qquad
\Re\int_{\gamma_3}\phi_2.
\]
By a direct computation, we have 
\[
v_2:=\Re\int_{\gamma_1}\phi_2
=2\int_{\lambda}^1\frac{b^2-t^2}{\sqrt{(t^2-\lambda^2)(1-t^2)(\lambda_1^2-t^2)(\lambda_2^2-t^2)}}dt>0. 
\]
Thus we have the period $\bm{v}_2=(0,v_2,0)$ in the direction parallel to the $x_2$-axis. 
We set 
\begin{equation}\label{eq:torus}
\mathbb{T}^2=\mathbb{R}^2/\Lambda, 
\quad
\Lambda=\{n_1\bm{v}_1+n_2\bm{v}_2:n_1,n_2\in\mathbb{Z}\}. 
\end{equation}
If we have
\begin{equation}\label{eq:period12}
\Re\int_{\gamma_2}\phi_1=0
\quad\text{and}\quad
\Re\int_{\gamma_3}\phi_2=0,
\end{equation}
then we can conclude that $X$ is well-defined in $\mathbb{T}^2\times\mathbb{R}$. 
It is easy to verify that the left equation in \eqref{eq:period12} holds if and only if
\begin{equation}\label{eq:period1}
\int_1^{\lambda_1}\phi_1(t)=0, 
\end{equation}
where 
\[
\phi_1(t)=\frac{t\big((1-\lambda-\lambda_1+\lambda_2)t^2
+\lambda\lambda_1-\lambda\lambda_2-\lambda_1\lambda_2+\lambda\lambda_1\lambda_2\big)}
{(t^2-a^2)\sqrt{(t^2-\lambda^2)(t^2-1)(\lambda_1^2-t^2)(\lambda_2^2-t^2)}}dt,
\]
and the right equation in \eqref{eq:period12} holds if and only if
\begin{equation}\label{eq:period2}
\int_{\lambda_1}^{\lambda_2}\phi_2(t)=0,
\end{equation}
where 
\[
\phi_2(t)=\frac{t^2-b^2}{\sqrt{(t^2-\lambda^2)(t^2-1)(t^2-\lambda_1^2)(\lambda_2^2-t^2)}}dt.
\] 

We fix $\lambda\in (0,1)$ and set
\begin{equation}
\xi_1=\xi_1(\lambda_1,\lambda_2)
=\int_{1}^{\lambda_1}\phi_1(t), 
\qquad
\xi_2=\xi_2(\lambda_1,\lambda_2)
=\int_{\lambda_1}^{\lambda_2}\phi_2(t).
\end{equation}

We prove the following:

\begin{proposition}\label{pr:period12}
We may choose an open interval $(a_1,a_2)\subset (0,1)$ satisfying the following 
property{\rm :}
for any $\lambda\in (a_1,a_2)$, there exist $\lambda_1$ and $\lambda_2$ with $1<\lambda_1<\lambda_2$ 
so that \eqref{eq:period1} and \eqref{eq:period2} hold. 
\end{proposition}

To prove this proposition, we first give six lemmas. 
These lemmas will be proven in the Appendix.  

\begin{figure}[htbp] 
\centering
 \includegraphics[height=2.3in]{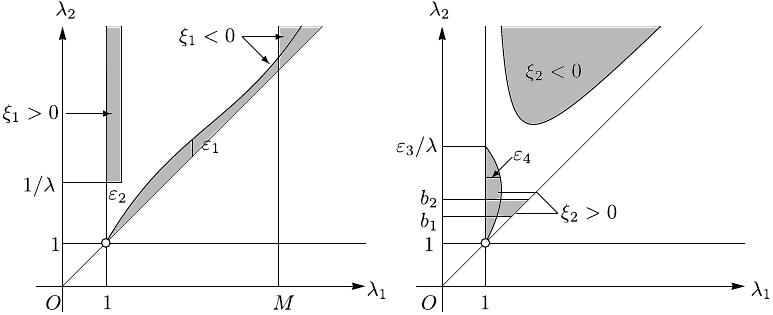} 
\caption{Signs of $\xi_1$ (left), and Signs of $\xi_2$ (right).}
\label{fig:xi1-xi2}
\end{figure} 

\begin{lemma}\label{lm:1}
For an arbitrary $\lambda_1>1$, there exists $\varepsilon_1=\varepsilon_1(\lambda_1)>0$ such that $\xi_1<0$ holds for 
$\lambda_1$, $\lambda_2$ satisfying $0<\lambda_2-\lambda_1 <\varepsilon_1$. 
See the left-hand side of Figure~$\ref{fig:xi1-xi2}$. 
\end{lemma}

\begin{lemma}\label{lm:2}
There exists $M>0$ such that $\xi_1<0$ holds for $\lambda_1$, $\lambda_2$ satisfying 
$\lambda_2>\lambda_1 >M$. 
See the left-hand side of Figure~$\ref{fig:xi1-xi2}$. 
\end{lemma}

\begin{lemma}\label{lm:3}
There exists $\varepsilon_2>0$ such that $\xi_1>0$ holds for 
an arbitrary $\lambda_1\,(>1)$ and $\lambda_2\,(>1/\lambda)$ satisfying
$0<\lambda_1-1 <\varepsilon_2$. 
See the left-hand side of Figure~$\ref{fig:xi1-xi2}$. 
\end{lemma}

\begin{lemma}\label{lm:4}
We may choose $\varepsilon_3>1$ satisfying the following property{\rm :} 
for $\lambda_2 < \varepsilon_3/\lambda$, there exists 
$\varepsilon_4=\varepsilon_4(\lambda_2)>0$ such that $\xi_2>0$ holds for $\lambda_1 \,(<\lambda_2)$ satisfying  
$0< \lambda_1-1 <\varepsilon_4$. 
See the right-hand side of Figure~$\ref{fig:xi1-xi2}$. 
\end{lemma}

\begin{lemma}\label{lm:5}
We may choose an open interval $(a_1,a_2)\subset (0,1)$ satisfying the following 
property{\rm :} for $\lambda\in (a_1,a_2)$, there exists 
$(b_1,b_2) \subset (1,1/\lambda]$ such that $\xi_2>0$ holds on 
$\{(\lambda_1,\lambda_2) \>|\> \lambda_2\in  (b_1,b_2), 1<\lambda_1<\lambda_2\}$. 
See the right-hand side of Figure~$\ref{fig:xi1-xi2}$. 
\end{lemma}

\begin{lemma}\label{lm:6}
For an arbitrary $\lambda_1>1$, there exists $N=N(\lambda_1)>0$ such that $\xi_2<0$ holds for $\lambda_2 \,(> N)$. 
See the right-hand side of Figure~$\ref{fig:xi1-xi2}$. 
\end{lemma}

We recall the following fact. 

\begin{fact}[Poincar\'e-Miranda Theorem {\cite{ku}}]\label{th:PM}
Let $f_i:[0,1]\times [0,1]\to\mathbb{R}^2$ $(i=1,2)$ be a continuous function such that 
$f_1>0$ on $x_1=0$, $f_1<0$ on $x_1=1$, and $f_2>0$ on $x_2=0$, $f_2<0$ on $x_2=1$. 
Then there exists $(x_1,x_2)\in [0,1]\times[0,1]$ such that $f_1(x_1,x_2)=f_2(x_1,x_2)=0$.
\end{fact}

Now we are ready to prove Proposition~\ref{pr:period12}.
 
\begin{proof}[Proof of Proposition~\ref{pr:period12}]
Choose any $\lambda\in(a_1,a_2)$, where $a_1$ and $a_2$ are as in Lemma~\ref{lm:5}. 
We define curves $c_1,\dots,c_6$ in the $(\lambda_1,\lambda_2)$-plane as follows. 
\begin{enumerate}
\item $c_1:=\{(\lambda_1,\varepsilon_1(\lambda_1)-\epsilon)\;:\;\lambda_1>1\}$, where $\varepsilon_1$ is given in Lemma~\ref{lm:1}. 
\item $c_2:=\{(M+\epsilon,\lambda_2)\;:\;\lambda_2>M\}$, where $M$ is given in Lemma~\ref{lm:2}. 
\item $c_3:=\{(\lambda_1,1/\lambda+\epsilon)\;:\;1<\lambda_1\le 1+\varepsilon_2\}\cup\{(1+\varepsilon_2,\lambda_2)\;:\;\lambda_2>1/\lambda\}$, where $\varepsilon_2$ is given in Lemma~\ref{lm:3}. 
\item $c_4:=\{(1+\varepsilon_4(\lambda_2)-\epsilon,\lambda_2)\;:\;1<\lambda_2<\varepsilon_3/\lambda\}$, where $\varepsilon_3$ and $\varepsilon_4$ are given in Lemma~\ref{lm:4}. 
\item $c_5:=\{(\lambda_1,b_2-\epsilon)\;:\;1<\lambda_1<\lambda_2\}$, where $b_2$ is given in Lemma~\ref{lm:5}. 
\item $c_6:=\{(\lambda_1,N(\lambda_1)+\epsilon)\;:\;\lambda_1>1\}$, where $N$ is given in Lemma~\ref{lm:6}. 
\end{enumerate}
And label $P_{ij}$ the intersection of $c_i$ and $c_j$. 
Here $\epsilon>0$ very small positive is chosen so that $\xi_1$ is positive on $P_{34}P_{36}$ (Lemma~\ref{lm:3}) and negative on $P_{15}P_{26}$ (Lemmas~\ref{lm:2} and \ref{lm:3}) and additionally $\xi_2$ is positive on $P_{34}P_{15}$ (Lemmas~\ref{lm:4} and \ref{lm:5}) and negative on $P_{26}P_{36}$ (Lemma~\ref{lm:6}). 
See Figure~\ref{fig:pm}. 

\begin{figure}[htbp] 
\centering
 \includegraphics[height=3in]{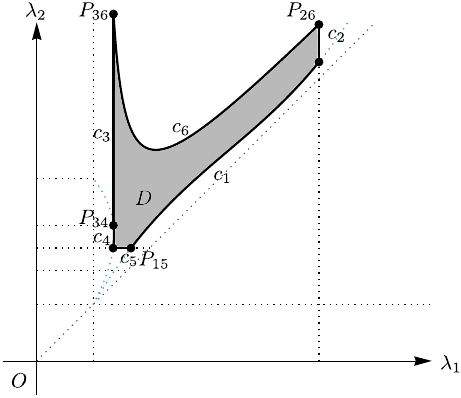} 
\caption{The domain $D$ in the $(\lambda_1,\lambda_2)$-plane.}
\label{fig:pm}
\end{figure} 

Let $D$ be the closed domain surrounded by the curves $c_1,\dots,c_6$.  
We compile our results for the sign of 
\[
(\xi_1,\xi_2)=\left(\int_{1}^{\lambda_1} \phi_1(t),\int_{\lambda_1}^{\lambda_2} \phi_2(t)\right)
\]
into $D$. 
Take a homeomorphism $\Phi$ on the unit square $[0,1]\times [0,1]$ into $D$ such that 
$(0,1)$, $(0,0)$, $(1,0)$, $(1,1)$ map to $P_{36}$, $P_{34}$, $P_{15}$, $P_{26}$, respectively.
We set $f_i:[0,1]\times [0,1]\to\mathbb{R}$ by $f_i=\xi_i\circ\Phi$ ($i=1,2$). 
Then by Fact~\ref{th:PM}, there exists $(x_1,x_2)\in [0,1]\times [0,1]$ such that 
$f_1(x_1,x_2)=f_2(x_1,x_2)=0$. 
Therefore, at $(\lambda_1,\lambda_2)=\Phi(x_1,x_2)$, we have 
\[
\xi_1(\lambda_1,\lambda_2)=
\xi_2(\lambda_1,\lambda_2)=0. 
\]
Thus the minimal surface exists for that $\lambda$, and because $\lambda$ was arbitrary, the minimal surface exists for all $\lambda\in(a_1,a_2)$.
\end{proof}

Summing up, we have proven Proposition \ref{pr:existence}.

\section{Embeddedness}
\label{embedded}

In this section we prove the following proposition. 

\begin{proposition}\label{pr:embedded}
For each $\lambda\in(a_1,a_2)$ and corresponding $\lambda_1$ and $\lambda_2$, the surface $M_{\lambda}$ given by Theorem \ref{thm1} is embedded.
\end{proposition}

To prove this, we first recall two facts, which are central to the arguments in this section. 

\begin{fact}[Schwarz Reflection Principle {\cite[Theorem 1.5.1]{ka6}}]\label{th:reflection}
Suppose a minimal surface $M$ contains in its boundary a curve $C$ that is either a straight line or a nonstraight planar geodesic. 
Then $M$ can be extended smoothly across $C$ by respectively $180^\circ$-rotation about $C$ or reflection through the plane containing $C$. 
\end{fact}

\begin{fact}[Krust {\cite[Theorem 2.4.1]{ka6}}, see also {\cite[Theorem 3.1]{dor}}]\label{th:krust}
If an embedded minimal surface $X:\D\to\R^3$, $\D=\{z\in\C:|z|<1\}$ can be written as a graph over a convex domain in a plane, then the conjugate surface $X^*:\D\to\R^3$ is also a graph over a domain in the same plane. 
\end{fact}

%
We also need the following lemma.

\begin{lemma}\label{lm:ReImw}
Let $R_1$ be the portion of the domain $R$ restricted to the first quadrant of the $z$-plane with $w(\infty)=+1$. 
Denote by $R_1^\circ$ the interior of $R_1$. 
For any $z\in R_1^\circ$, we have
\[
\Re(w)>0>\Im(w). 
\]
\end{lemma}

\begin{proof}
In the proof, we always consider $\arg z$ as a single valued function that takes value 
in the interval $(-\pi,\pi]$. 
We first show that the imaginary part of 
\[
w^2=\frac{g_1(z)}{g_2(z)}
   =\frac{(z-\lambda)(z+1)(z-\lambda_1)(z+\lambda_2)}{(z+\lambda)(z-1)(z+\lambda_1)(z-\lambda_2)}
\]
is negative for $z\in R_1^\circ$. 
The map $\mathbb{C}\cup\{\infty\}\ni z\mapsto g_1(z)/g_2(z)\in\mathbb{C}\cup\{\infty\}$ is degree 4 
and the restriction $\mathbb{R}\cup\{\infty\}\ni x\mapsto g_1(x)/g_2(x)\in\mathbb{R}\cup\{\infty\}$ 
is also four to one map. 
Thus $w^2\in\mathbb{R}\cup\{\infty\}$ if and only if $z\in\mathbb{R}\cup\{\infty\}$. 
Hence, by continuity, $\Im(w^2)$ is either positive or negative for all $z\in R_1^\circ$. 
Let $z_0=r+i\epsilon\in R_1^\circ$, where $r$ is sufficiently large positive number and 
$\epsilon$ is sufficiently small positive number such that both $g_1(z_0)$ and $g_2(z_0)$ take 
values in $R_1^\circ$. 
Then there exist $\theta_1,\theta_2\in (0,\pi/2)$ so that 
\[
\arg g_i(z_0) = \theta_i\qquad (i=1,2). 
\]
Then, since 
\begin{align*}
0&<\arg (z_0+\lambda_2)
  <\arg (z_0+\lambda_1)
  <\arg (z_0+1)
  <\arg (z_0+\lambda) \\
 &<\arg (z_0-\lambda)
  <\arg (z_0-1)
  <\arg (z_0-\lambda_1)
  <\arg (z_0-\lambda_2)
  <\frac{\pi}{2}, 
\end{align*}
we have 
\begin{align*}
\theta_1
&=\arg (z_0-\lambda)+\arg (z_0+1)+\arg (z_0-\lambda_1)+(z_0+\lambda_2) \\
&<\arg (z_0-1)+\arg (z_0+\lambda)+\arg (z_0-\lambda_2)+(z_0+\lambda_1)
 =\theta_2,  
\end{align*}
and hence $\arg w^2=\theta_1-\theta_2<0$. 
This implies that $\Im(w^2)<0$. 
Therefore, $w$ satisfies either
\[
\Re(w)<0<\Im(w)
\qquad\mbox{or}\qquad
\Re(w)>0>\Im(w). 
\]
Since $w(\infty)=+1$, we see that $\Re(w)>0>\Im(w)$. 
\end{proof}
 
\begin{proof}[Proof of Proposition~\ref{pr:embedded}]
The surface $M_{\lambda}$ has symmetry planes parallel to the three coordinate planes.  
Thus, $M_{\lambda}$ is made up of eight congruent pieces.  
Hence, we only need to show that one of these eight pieces is embedded.  
Let $R_1$ be the portion of $R$ as in Lemma~\ref{lm:ReImw}. 
We set $\Omega_1:=X(R_1)$. 
See Figure~\ref{fp-fpc}, left. 

\begin{figure}[htbp] 
\centering
 \includegraphics[width=2.2in]{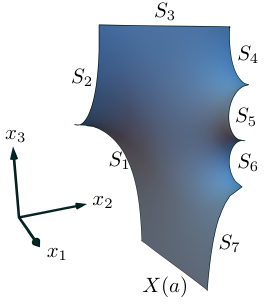} \ \ 
 \includegraphics[width=2.8in]{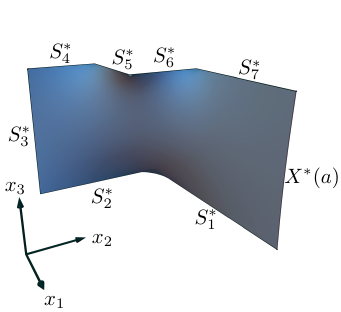} 
\caption{Fundamental piece $\Omega_1$ (left) and its conjugate $\Omega_1^*$ (right).}
\label{fp-fpc}
\end{figure} 

We set $\Pi=\{(x_1,x_2,x_3)\in\mathbb{R}^3\;:\;x_1=x_2\}$. 
We will show that $\Omega_1$ is a graph over a domain in $\Pi$.
By Lemma~\ref{lm:sym} and Fact~\ref{th:reflection},
the boundary of $\Omega_1$ consists of seven planar symmetry curves:
\begin{enumerate}
\item
$s_1$ is the real interval $[\lambda,a)$ and $S_1:=X(s_1)$ lies in a plane parallel to the $(x_2,x_3)$-plane;
\item
$s_2$ is the real interval $[0,\lambda]$ and $S_2:=X(s_2)$ lies in a plane parallel to the $(x_1,x_3)$-plane;
\item
$s_3$ is the positive imaginary axis and $S_3:=X(s_3)$ lies in a plane parallel to the $(x_1,x_2)$-plane;
\item
$s_4$ is the real interval $[\lambda_2,\infty)$ and $S_4:=X(s_4)$ lies in a plane parallel to the $(x_1,x_3)$-plane different than the plane containing $S_2$;
\item
$s_5$ is the real interval $[\lambda_1,\lambda_2]$ and $S_5:=X(s_5)$ lies in a plane parallel to the $(x_2,x_3)$-plane different than the plane containing $S_1$;
\item
$s_6$ is the real interval $[1,\lambda_1]$ and $S_6:=X(s_6)$ lies in plane containing $S_4$;
\item
$s_7$ is the real interval $(a,1]$ and $S_7:=X(s_7)$ lies in plane containing $S_5$.
\end{enumerate}

Take the conjugate $\Omega_1^*:=X^*(R_1)$ of $\Omega_1$. 
See Figure~\ref{fp-fpc}, right.
By Fact~\ref{th:krust}, it suffices to prove that $\Omega_1^*$ is a graph over a domain in $\Pi$.
By Remark~\ref{rm:conj}, 
the boundary of $\Omega_1^*$ consists of seven straight lines:
\begin{enumerate}
\item
$S_1^*:=X^*(s_1)$ is a straight line parallel to the $x_1$-axis;
\item
$S_2^*:=X^*(s_2)$ is a straight segment parallel to the $x_2$-axis;
\item
$S_3^*:=X^*(s_3)$ is a straight segment parallel to the $x_3$-axis;
\item
$S_4^*:=X^*(s_4)$ is a straight segment parallel to the $x_2$-axis;
\item
$S_5^*:=X^*(s_5)$ is a straight segment parallel to the $x_1$-axis;
\item
$S_6^*:=X^*(s_6)$ is a straight segment parallel to the $x_2$-axis;
\item
$S_7^*:=X^*(s_7)$ is a straight line parallel to the $x_1$-axis. 
\end{enumerate}

It is easy to verify that the projection of the boundary of $\Omega_1^*$ onto $\Pi$ is embedded.  
By Lemma~\ref{lm:ReImw}, $\Re(w)>0>\Im(w)$ in the interior of $R_1$.  
Thus the projection of $\Omega_1^*$ to $\Pi$ is a submersion.  
Hence, $\Omega_1^*$ is a graph over $\Pi$, proving that $M_{\lambda}$ is embedded.
\end{proof}

\appendix

\section{Proofs of Lemmas \ref{lm:1}--\ref{lm:6}}

Throughout this appendix, we refer to formulae in \cite{bf}.  
For example, the sentence ``\textbf{[400.00]}'' stands for ``the formula \textbf{[400.00]} in \cite{bf}''. 

In advance, we introduce the elliptic integrals of the first kind, the second kind, and the third kind: 
\begin{align*}
& F (\varphi, k) =\int_0^{\varphi} \frac{d\theta}{\sqrt{1-k^2\sin^2\theta}}, \quad 
E(\varphi, k) = \int_0^{\varphi} \sqrt{1-k^2\sin^2\theta}d\theta, \\
& \Pi (\varphi, \alpha^2,k) = \int_0^{\varphi} \frac{d\theta}{(1-\alpha^2\sin^2\theta)\sqrt{1-k^2\sin^2\theta}}. 
\end{align*}
In particular, for $\varphi=\pi/2$, we call them the complete elliptic integrals of the first kind, 
the second kind, and the third kind, denoted by 
$K (k)$, $E (k)$, $\Pi (\alpha^2,k)$.

\subsection{Estimate for $\xi_1$}

Setting $s=t^2$, we have  
\[
\xi_1 
= \frac{1-\lambda-\lambda_1+\lambda_2}{2} \int_1^{\lambda_1^2} \frac{s+\frac{\lambda\lambda_1-\lambda\lambda_2-\lambda_1\lambda_2+\lambda\lambda_1\lambda_2}{1-\lambda-\lambda_1+\lambda_2} }{(s-a^2)\sqrt{(s-\lambda^2)(s-1)(\lambda_1^2-s)(\lambda_2^2-s)}}ds. 
\]

By  \textbf{[254.40]}, we have
\[
\xi_1 = \frac{(1-\lambda)(\lambda_1-1)(\lambda_2+1)}{(a^2-1)\sqrt{(\lambda_2^2-1)(\lambda_1^2-\lambda^2)}}
\int_0^{K(k)} \frac{1-\alpha_2^2 {\rm sn}^2 u}{1-\alpha_3^2 {\rm sn}^2 u} du, 
\]
where 
\[
\alpha_2^2 = \frac{(\lambda_2-\lambda)(\lambda_1+1)}{(\lambda_1-\lambda)(\lambda_2+1)}>1, \qquad
\alpha_3^2 = \frac{(\lambda_1^2-1)(a^2-\lambda^2)}{(\lambda_1^2-\lambda^2)(a^2-1)}<0, \qquad 
k^2 = \frac{(\lambda_1^2-1)(\lambda_2^2-\lambda^2)}{(\lambda_2^2-1)(\lambda_1^2-\lambda^2)}. 
\]
Note that 
\[
(\lambda_1^2-1)(\lambda_2^2-\lambda^2)-(\lambda_2^2-1)(\lambda_1^2-\lambda^2)
= -(1-\lambda^2)(\lambda_2^2-\lambda_1^2)<0
\]
implies $0<k<1$. It follows from \textbf{[340.01]} that  
\[
\int_0^{K(k)} \frac{1-\alpha_2^2 {\rm sn}^2 u}{1-\alpha_3^2 {\rm sn}^2 u} du 
= \frac{1}{\alpha_3^2} [ (\alpha_3^2-\alpha_2^2) \Pi (\alpha_3^2,k)+\alpha_2^2 K (k)]. 
\]
Hence, we have 
\begin{align*}
\xi_1 = \frac{(1-\lambda)(\lambda_1+\lambda)}{(\lambda_1+1)(a^2-\lambda^2)}
\sqrt{\frac{(\lambda_2+1)(\lambda_1-\lambda)}{(\lambda_2-1)(\lambda_1+\lambda)}}
[ (\alpha_3^2-\alpha_2^2) \Pi (\alpha_3^2,k)+\alpha_2^2 K (k)]. 
\end{align*}
Since $\alpha_3^2<0$ holds, it corresponds to Case~I in \textbf{[400.01]}. 

By means of \textbf{[410.01]} and Heuman's Lambda function $\Lambda_0 (\psi,k)$, we find 
\begin{equation}\label{period_trans_2}
\xi_1 
=\frac{(1-\lambda)(\lambda_1+\lambda)}{(\lambda_1+1)(a^2-\lambda^2)}
\sqrt{\frac{(\lambda_2+1)(\lambda_1-\lambda)}{(\lambda_2-1)(\lambda_1+\lambda)}} 
\left[ \frac{\alpha_3^2(k^2-\alpha_2^2)}{k^2-\alpha_3^2} K(k) -\frac{\pi\alpha_3^2(\alpha_3^2-\alpha_2^2)\Lambda_0 (\psi,k)}{2\sqrt{(-\alpha_3^2) (1-\alpha_3^2)(k^2-\alpha_3^2)}}\right],  
\end{equation}
where $\sin\psi = \sqrt{\alpha_3^2/(\alpha_3^2-k^2)}$. 

\subsubsection{Proof of Lemma~\ref{lm:1}} 

For an arbitrary fixed $\lambda_1>1$, we consider the sign of the large square brackets in \eqref{period_trans_2} as $\lambda_2\to \lambda_1$. 
Remark that 
\[
a^2-1 = \frac{\alpha-2-\sqrt{\alpha^2-4\lambda\lambda_1\lambda_2}}{2}
= \frac{2(\lambda-1)(\lambda_1-1)(\lambda_2+1)}{\alpha-2+\sqrt{\alpha^2-4\lambda\lambda_1\lambda_2}}. 
\]
In this case, we have $k\to 1$ and $a^2\to \lambda$. 
Straightforward calculations yield 
\begin{align*}
k^2-\alpha_2^2 &= -\frac{(1+\lambda)(\lambda_1+1)(\lambda_2-\lambda)(\lambda_2-\lambda_1)}{(\lambda_1^2-\lambda^2)(\lambda_2^2-1)}, \\
k^2-\alpha_3^2 &= \frac{(1-\lambda^2)(\lambda_1^2-1)(\lambda_2^2-a^2)}{(1-a^2)(\lambda_1^2-\lambda^2)(\lambda_2^2-1)}, \quad 
1-k^2 = \frac{(1-\lambda^2)(\lambda_2^2-\lambda_1^2)}{(\lambda_1^2-\lambda^2)(\lambda_2^2-1)}. 
\end{align*}
By virtue of \textbf{[900.05]}, we find 
\[
\lim_{k\to 1} \frac{K(k)}{-\log \sqrt{1-k^2}}=1. 
\]
As a consequence, we have 
\begin{align*}
\frac{\alpha_3^2(k^2-\alpha_2^2)}{k^2-\alpha_3^2} K(k) 
& = \frac{(a^2-\lambda^2)(\lambda_1+1)(\lambda_2-\lambda)\sqrt{(\lambda_2^2-1)(\lambda_2-\lambda_1)}}{(1-\lambda)(\lambda_2^2-a^2)\sqrt{(\lambda_1^2-\lambda^2)(\lambda_2+\lambda_1)}}\cdot 
\sqrt{\frac{\lambda_2^2-\lambda_1^2}{(\lambda_1^2-\lambda^2)(\lambda_2^2-1)}} \\
&\qquad \qquad \qquad \qquad \times \log \sqrt{\frac{(1-\lambda^2)(\lambda_2^2-\lambda_1^2)}{(\lambda_1^2-\lambda^2)(\lambda_2^2-1)}}\cdot \frac{K(k)}{\log \sqrt{1-k^2}}
\to 0
\end{align*}
as $\lambda_2\to\lambda_1$.
Moreover, we find 
\begin{align*}
\frac{\pi\alpha_3^2(\alpha_3^2-\alpha_2^2)\Lambda_0 (\psi,k)}{2\sqrt{(-\alpha_3^2) (1-\alpha_3^2)(k^2-\alpha_3^2)}} 
&  =\frac{\pi}{2(1-\lambda^2)} \sqrt{\frac{(1-a^2)(a^2-\lambda^2)(\lambda_1^2-\lambda^2)(\lambda_2^2-1)}{(\lambda_1^2-a^2)(\lambda_2^2-a^2)}} \\
&\qquad \qquad \qquad \times \left[ \frac{(\lambda_1^2-1)(a^2-\lambda^2)}{(\lambda_1^2-\lambda^2)(1-a^2)}+\frac{(\lambda_2-\lambda)(\lambda_1+1)}{(\lambda_1-\lambda)(\lambda_2+1)}  \right] \Lambda_0 (\psi,k). 
\end{align*}
It follows from \textbf{[151.01]} that $\Lambda_0 (\psi,1) = 2\psi/\pi$, and thus the above is positive as 
$\lambda_2\to\lambda_1$, which completes the proof. 

\subsubsection{Proof of Lemma~\ref{lm:2}} 

We consider the sign of the large square brackets in \eqref{period_trans_2} as $\lambda_1 \to \infty$. 
In this case, the inequality $\lambda_2>\lambda_1$ implies that $\lambda_2\to \infty$ and $k\to 1$. 
Note that 
\begin{align*}
& \frac{\alpha}{\lambda_1 \lambda_2} = 1+\frac{\lambda+\lambda_1-\lambda_2-\lambda\lambda_1+\lambda\lambda_2}{\lambda_1 \lambda_2} \to 1, \\
& a^2 = \frac{2 \lambda\lambda_1\lambda_2}{\alpha+\sqrt{\alpha^2-4\lambda\lambda_1\lambda_2}} 
= \frac{2 \lambda}{\frac{\alpha}{\lambda_1\lambda_2}+\sqrt{\frac{\alpha^2}{\lambda_1^2\lambda_2^2}-\frac{4\lambda}{\lambda_1\lambda_2}}} 
\to \lambda. 
\end{align*}
In general, we have to consider the orders of $\lambda_1$ and $\lambda_2$ 
as $\lambda_1, \lambda_2\to \infty$. 
In particular, for the following estimates, either $1-\lambda_1^2/\lambda_2^2 \to 0$ or $1-\lambda_1^2/\lambda_2^2$ can be estimated by a finite positive value ($<1$). 
We see that 
\begin{align*}
\frac{\alpha_3^2(k^2-\alpha_2^2)}{k^2-\alpha_3^2} K(k) 
& = \frac{(a^2-\lambda^2)(1+\frac{1}{\lambda_1})(1-\frac{\lambda}{\lambda_2})\sqrt{(1-\frac{1}{\lambda_2^2})(1-\frac{\lambda_1}{\lambda_2})}}{(1-\lambda)(1-\frac{a^2}{\lambda_2^2})\sqrt{(1-\frac{\lambda^2}{\lambda_1^2})(1+\frac{\lambda_1}{\lambda_2})}}\cdot 
\left[\frac{1-\frac{\lambda_1^2}{\lambda_2^2}}{(\lambda_1^2-\lambda^2)(1-\frac{1}{\lambda_2^2})}\right]^{1/2} \\
&\qquad \qquad \qquad \qquad \times \log \left[\frac{(1-\lambda^2)(1-\frac{\lambda_1^2}{\lambda_2^2})}{(\lambda_1^2-\lambda^2)(1-\frac{1}{\lambda_2^2})}\right]^{1/2}\cdot \frac{K(k)}{\log \sqrt{1-k^2}} 
\to 0 
\end{align*}
as $\lambda_1\to\infty$.
Moreover, we find 
\begin{align*}
\frac{\pi\alpha_3^2(\alpha_3^2-\alpha_2^2)\Lambda_0 (\psi,k)}{2\sqrt{(-\alpha_3^2) (1-\alpha_3^2)(k^2-\alpha_3^2)}}
& =\frac{\pi}{2(1-\lambda^2)} \left[\frac{(1-a^2)(a^2-\lambda^2)(1-\frac{\lambda^2}{\lambda_1^2})(1-\frac{1}{\lambda_2^2})}{(1-\frac{a^2}{\lambda_1^2})(1-\frac{a^2}{\lambda_2^2})}\right]^{1/2} \\
& \qquad \qquad \times \left[ \frac{(1-\frac{1}{\lambda_1^2})(a^2-\lambda^2)}{(1-\frac{\lambda^2}{\lambda_1^2})(1-a^2)}+\frac{(1-\frac{\lambda}{\lambda_2})(1+\frac{1}{\lambda_1})}{(1-\frac{\lambda}{\lambda_1})(1+\frac{1}{\lambda_2})}  \right] \Lambda_0 (\psi,k). 
\end{align*}
It is easy to verify that the limit of the above as $\lambda_1\to\infty$ is positive. 
Therefore, the lemma follows.  

\subsubsection{Proof of Lemma~\ref{lm:3}} 

For an arbitrary fixed $\lambda_2(>1/\lambda)$, we consider the sign of the large square brackets in \eqref{period_trans_2} as $\lambda_1 \to 1$. In this case, $k\to 0$ holds. 
Remark that $\alpha\to 1+\lambda\lambda_2$, and 
\begin{align*}
 a^2 &\to \frac{1+\lambda\lambda_2-|\lambda\lambda_2-1|}{2}
=1, \\
 \frac{\lambda_1^2-1}{1-a^2} &\to \frac{\lambda\lambda_2-1+|\lambda\lambda_2-1|}{(1-\lambda)(\lambda_2+1)}
=\frac{2(\lambda\lambda_2-1)}{(1-\lambda)(\lambda_2+1)}, \\
 \frac{1-a^2}{\lambda_1^2-a^2} 
& = \frac{2(1-\lambda)(\lambda_2+1)}{[\alpha-2+\sqrt{\alpha^2-4\lambda\lambda_1\lambda_2}]\bigg[\lambda_1+1+\frac{2(1-\lambda)(\lambda_2+1)}{\alpha-2+\sqrt{\alpha^2-4\lambda\lambda_1\lambda_2}}\bigg] } \\
& \to \frac{(1-\lambda)(\lambda_2+1)}{[\lambda\lambda_2-1+|\lambda\lambda_2-1|]
\bigg[1+\frac{(1-\lambda)(\lambda_2+1)}{\lambda\lambda_2-1+|\lambda\lambda_2-1|}\bigg]}
=\frac{(1-\lambda)(\lambda_2+1)}{(1+\lambda)(\lambda_2-1)}.
\end{align*}
By means of \textbf{[111.02]}, \textbf{[151.01]}, 
the large square brackets in \eqref{period_trans_2} as $\lambda_1 \to 1$ become  
\begin{align*}
& \frac{(a^2-\lambda^2)(\lambda_1+1)(\lambda_2-\lambda)(\lambda_2-\lambda_1)}{(1-\lambda)(\lambda_1^2-\lambda^2)(\lambda_2^2-a^2)} K (k) 
-\frac{\pi}{2(1-\lambda^2)} \sqrt{\frac{(1-a^2)(a^2-\lambda^2)(\lambda_1^2-\lambda^2)(\lambda_2^2-1)}{(\lambda_1^2-a^2)(\lambda_2^2-a^2)}} \\
& \hspace{210pt} \times 
\left[ \frac{(\lambda_1^2-1)(a^2-\lambda^2)}{(\lambda_1^2-\lambda^2)(1-a^2)}+\frac{(\lambda_2-\lambda)(\lambda_1+1)}{(\lambda_1-\lambda)(\lambda_2+1)}  \right] \Lambda_0 (\psi,k) \\
&\to 
\frac{\pi (a^2-\lambda^2)(\lambda_2-1)}{(1-\lambda)(1-\lambda^2)(\lambda_2^2-a^2)} 
\left[ \lambda_2-\lambda
-\frac{\lambda_2+1}{2} \sqrt{\frac{1-a^2}{\lambda_1^2-a^2}} 
\left(
\frac{(a^2-\lambda^2)(\lambda_1^2-1)}{(1+\lambda)(1-a^2)} + \frac{2(\lambda_2-\lambda)}{\lambda_2+1} 
\right)
\right] \\
& 
=\frac{\pi}{(1-\lambda)(\lambda_2+1)} (\lambda_2-\lambda-\sqrt{(1-\lambda^2)(\lambda_2^2-1)}).  
\end{align*}
Since 
\[
(\lambda_2-\lambda)^2-(1-\lambda^2)(\lambda_2^2-1) = (\lambda\lambda_2-1)^2 \geq 0
\]
holds, the above is positive provided that $\lambda\lambda_2>1$. 
Note that, since $\lambda\lambda_2>1$, we have $\lambda_2 \nrightarrow 1$. 
As a result, we may choose $\lambda_1$ which does not depend on $\lambda_2$. 
Therefore, we obtain the lemma. 

\subsection{Estimate for $\xi_2$ from below}

Setting $s=t^2$, we have 
\[
\xi_2 
= -\frac{1}{2}
\int_{\lambda_1^2}^{b^2} \frac{b^2-s}{\sqrt{s(s-\lambda^2)(s-1)(s-\lambda_1^2)(\lambda_2^2-s)}}ds  
 + \frac{1}{2}\int_{b^2}^{\lambda_2^2} \frac{s-b^2}{\sqrt{s(s-\lambda^2)(s-1)(s-\lambda_1^2)(\lambda_2^2-s)}} ds. 
\]

\subsubsection{Proof of Lemma~\ref{lm:4}} 

For an arbitrary fixed $\lambda_2\in (1,1/\lambda]$, we consider the case $\lambda_1\to 1$. 
We first have 
\[
\int_{\lambda_1^2}^{b^2} \frac{b^2-s}{\sqrt{s(s-\lambda^2)(s-1)(s-\lambda_1^2)(\lambda_2^2-s)}}ds\\
 \leq 
\frac{\lambda_1}{\sqrt{\lambda_1^2-\lambda^2}}
\int_{\lambda_1^2}^{b^2} \frac{b^2- s}{s\sqrt{(s-1)(s-\lambda_1^2)(\lambda_2^2-s)}} ds. 
\]
By \textbf{[235.18]} and \textbf{[340.01]}, we have
\[
\int_{\lambda_1^2}^{b^2} \frac{b^2- s}{s\sqrt{(s-1)(s-\lambda_1^2)(\lambda_2^2-s)}} ds 
 = \frac{2}{\lambda_1^2 \sqrt{\lambda_2^2-1}} \left[ \lambda_1^2 (b^2-1) F (\varphi,k) 
- b^2 (\lambda_1^2-1) \Pi (\varphi,\alpha^2,k) \right], 
\]
where 
\[
k^2 = \frac{\lambda_2^2-\lambda_1^2}{\lambda_2^2-1}, \qquad
\sin \varphi = \sqrt{\frac{(\lambda_2^2-1)(b^2-\lambda_1^2)}{(\lambda_2^2-\lambda_1^2)(b^2-1)}}=\frac{1}{k} \sqrt{\frac{b^2-\lambda_1^2}{b^2-1}}, \qquad
\alpha^2 = \frac{ \lambda_2^2-\lambda_1^2 }{\lambda_1^2 (\lambda_2^2 -1)}
=\frac{k^2}{\lambda_1^2}<k^2. 
\]
Thus, we see that  
\begin{align}
\nonumber 0 & \leq \frac{1}{2}\int_{\lambda_1^2}^{b^2} \frac{b^2-s}{\sqrt{s(s-\lambda^2)(s-1)(s-\lambda_1^2)(\lambda_2^2-s)}}ds \\
& \leq 
\frac{\lambda_1^2 (b^2-1) F (\varphi,k) 
- b^2 (\lambda_1^2-1) \Pi (\varphi,\alpha^2,k) }{\lambda_1 \sqrt{(\lambda_2^2-1)(\lambda_1^2-\lambda^2)}}
\leq 
\frac{\lambda_1 (b^2-1) F (\varphi,k)}{\sqrt{(\lambda_2^2-1)(\lambda_1^2-\lambda^2)}}. \label{key_estimates_1}
\end{align}
We shall show that, for $\lambda_2 \,(\leq 1/\lambda)$, the right-hand side converges to $0$ 
as $\lambda_1\to 1$. 
In this case, $k\to 1$ holds. In order to do, we investigate the limit of $b^2-1$. 
For $\lambda\lambda_2=1$, we have 
\[
 b^2-1 
 =\frac{\sqrt{\lambda_1-1} [ \sqrt{\lambda_1-1} (1-\lambda+\lambda_2)+\sqrt{\lambda_1(1-\lambda+\lambda_2)^2-(1+\lambda-\lambda_2)^2} ]}{2}. 
\]
For $\lambda\lambda_2< 1$, we have 
\begin{align*}
b^2 &= \frac{1}{2} \bigg[1+\lambda\lambda_2 + (\lambda_1-1)(1-\lambda+\lambda_2) 
+ \bigg((1-\lambda\lambda_2)^2 +(\lambda_1-1) [ (\lambda_1-1)(1-\lambda+\lambda_2)^2 \\
& \hspace{200pt} -2 (-1+\lambda+(-1+\lambda+\lambda^2) \lambda_2-\lambda\lambda_2^2)]\bigg)^{1/2}\, \bigg] \\
&= \frac{1+\lambda\lambda_2 + |1-\lambda\lambda_2|}{2} +(\lambda_1-1) (\cdots) 
 = 1 +(\lambda_1-1) (\cdots)
\end{align*}
It follows from $1-k^2 = (\lambda_1^2-1)/(\lambda_2^2-1)$ that, for $\lambda \lambda_2 \leq 1$, we find 
\begin{align*}
0 & \leq (b^2-1) F (\varphi,k) \\
& \leq (b^2-1) K (k) 
=
\begin{cases} 
\sqrt{\frac{\lambda_2^2-1}{\lambda_1+1}} (\cdots)\sqrt{1-k^2} K (k) & (\lambda\lambda_2=1) \\
\frac{(\lambda_2^2-1) (\cdots)}{\lambda_1+1} (1-k^2) K (k) & (\lambda\lambda_2<1) 
\end{cases} 
\to 0.  
\end{align*}
We next have 
\[
\int_{b^2}^{\lambda_2^2} \frac{s-b^2}{\sqrt{s(s-\lambda^2)(s-1)(s-\lambda_1^2)(\lambda_2^2-s)}}ds 
\geq 
\frac{\lambda_2}{\sqrt{\lambda_2^2-\lambda^2}} \int_{b^2}^{\lambda_2^2} \frac{s-b^2}{s\sqrt{(s-1)(s-\lambda_1^2)(\lambda_2^2-s)}}ds. 
\]
By  \textbf{[236.18]} and \textbf{[340.01]}, we have 
\[
\int_{b^2}^{\lambda_2^2} \frac{s-b^2}{s\sqrt{(s-1)(s-\lambda_1^2)(\lambda_2^2-s)}}ds 
 = \frac{2}{\lambda_2^2 \sqrt{\lambda_2^2-1}} \left[ \lambda_2^2 F (\psi,k) 
- b^2 \Pi (\psi,(\alpha^{\prime})^2,k) \right], 
\]
where 
\[
k^2 = \frac{\lambda_2^2-\lambda_1^2}{\lambda_2^2-1}, \qquad 
\sin \psi = \sqrt{\frac{\lambda_2^2-b^2}{\lambda_2^2-\lambda_1^2}}
 =\frac{1}{k} \sqrt{\frac{\lambda_2^2-b^2}{\lambda_2^2-1}}, \qquad
 (\alpha^{\prime})^2 = \frac{ \lambda_2^2-\lambda_1^2 }{\lambda_2^2}<k^2. 
\]
As a consequence, we have 
\begin{equation}\label{elliptic_trans}
\frac{1}{2}\int_{b^2}^{\lambda_2^2} \frac{s-b^2}{\sqrt{s(s-\lambda^2)(s-1)(s-\lambda_1^2)(\lambda_2^2-s)}}ds 
\geq 
\frac{\lambda_2^2 F (\psi,k) 
- b^2 \Pi (\psi,(\alpha^{\prime})^2,k)}{\lambda_2 \sqrt{(\lambda_2^2-1)(\lambda_2^2-\lambda^2)}}. 
\end{equation}
By means of \eqref{key_estimates_1}, \eqref{elliptic_trans}, we find 
\[
\xi_2 \geq -\frac{\lambda_1^2 (b^2-1) F (\varphi,k) 
- b^2 (\lambda_1^2-1) \Pi (\varphi,\alpha^2,k) }{\lambda_1 \sqrt{(\lambda_2^2-1)(\lambda_1^2-\lambda^2)}} 
 + \frac{\lambda_2^2 F (\psi,k) 
- b^2 \Pi (\psi,(\alpha^{\prime})^2,k) }{\lambda_2 \sqrt{(\lambda_2^2-1)(\lambda_2^2-\lambda^2)}}.
\]
Since $\lambda\lambda_2\leq 1$, the first term of the right-hand side converges to $0$ as $\lambda_1\to 1$. 
It follows from \textbf{[111.04]} that the second term of the right-hand side converses to 
\[
 \frac{\lambda_2^2 \log (\tan\psi+\frac{1}{\cos\psi})
- \frac{1}{1-(\alpha^{\prime})^2} [\log (\tan\psi+\frac{1}{\cos\psi}) - \frac{\alpha^{\prime}}{2} \log \frac{1+\alpha^{\prime} \sin\psi}{1-\alpha^{\prime} \sin\psi}]}{\lambda_2 \sqrt{(\lambda_2^2-1)(\lambda_2^2-\lambda^2)}} 
 = \frac{\frac{\alpha^{\prime}}{2}\lambda_2 \log \frac{1+\alpha^{\prime}}{1-\alpha^{\prime}}}{\sqrt{(\lambda_2^2-1)(\lambda_2^2-\lambda^2)}}>0,
\]
which completes the proof. 

\subsubsection{Proof of Lemma~\ref{lm:5}} 

We first have 
\begin{align*}
\xi_2 
&= \frac{1}{2}
\int_{\lambda_1^2}^{\lambda_2^2} \frac{s-\lambda_1^2}{\sqrt{s(s-\lambda^2)(s-1)(s-\lambda_1^2)(\lambda_2^2-s)}}ds \\ 
&\hspace{100pt}
 - \frac{b^2-\lambda_1^2}{2}\int_{\lambda_1^2}^{\lambda_2^2} \frac{ds}{\sqrt{s(s-\lambda^2)(s-1)(s-\lambda_1^2)(\lambda_2^2-s)}} \\
& \geq 
\frac{\lambda_2}{2 \sqrt{\lambda_2^2-\lambda^2}}
\int_{\lambda_1^2}^{\lambda_2^2} \frac{s-\lambda_1^2}{s\sqrt{(s-1)(s-\lambda_1^2)(\lambda_2^2-s)}}ds \\ 
&\hspace{100pt} 
- \frac{\lambda_1(b^2-\lambda_1^2)}{2\sqrt{\lambda_1^2-\lambda^2}}\int_{\lambda_1^2}^{\lambda_2^2} \frac{ds}{s\sqrt{(s-1)(s-\lambda_1^2)(\lambda_2^2-s)}}. 
\end{align*}
By \textbf{[236.02]}, \textbf{[236.18]},
we have 
\begin{align*}
& \int_{\lambda_1^2}^{\lambda_2^2} \frac{s-\lambda_1^2}{s\sqrt{(s-1)(s-\lambda_1^2)(\lambda_2^2-s)}}ds 
= \frac{2}{\lambda_2^2\sqrt{\lambda_2^2-1}} [-\lambda_1^2 \Pi (\alpha^2,k) +\lambda_2^2 K (k) ], \\
& \int_{\lambda_1^2}^{\lambda_2^2} \frac{ds}{s\sqrt{(s-1)(s-\lambda_1^2)(\lambda_2^2-s)}} 
= \frac{2}{\lambda_2^2\sqrt{\lambda_2^2-1}} \Pi (\alpha^2,k), 
\end{align*}
where 
\[
k^2 = \frac{\lambda_2^2-\lambda_1^2}{\lambda_2^2-1}, \quad
\alpha^2 = \frac{\lambda_2^2-\lambda_1^2}{\lambda_2^2}<k^2.
\]
Hence, we have 
\begin{equation}\label{key_6}
\xi_2 
\geq \frac{1}{\lambda_2^2\sqrt{\lambda_2^2-1}}
\Bigg[ 
-\lambda_1 \left( \frac{\lambda_1\lambda_2}{\sqrt{\lambda_2^2-\lambda^2}} +\frac{b^2-\lambda_1^2}{\sqrt{\lambda_1^2-\lambda^2}}\right) \Pi (\alpha^2,k) \\
+\frac{\lambda_2^3}{\sqrt{\lambda_2^2-\lambda^2}} K (k) 
\Bigg]. 
\end{equation}
Since the integrand of the left-hand side in \textbf{[362.10]} is non-negative, 
we find 
\[
[k^2+\alpha^2 (1-k^2)] K (k) \geq \alpha^2 E (k) + (1-\alpha^2) k^2 \Pi (\alpha^2,k). 
\]
Similarly, by means of \textbf{[362.07]}, we have $E(k) \geq (1-k^2) \Pi (\alpha^2,k)$. 
Consequently, we find 
\[
K (k) \geq \frac{k^2+\alpha^2 -2\alpha^2 k^2}{k^2+\alpha^2 -\alpha^2 k^2} \Pi (\alpha^2,k). 
\]
Hence, we have 
\begin{align*}
& -\lambda_1 \left( \frac{\lambda_1\lambda_2}{\sqrt{\lambda_2^2-\lambda^2}} +\frac{b^2-\lambda_1^2}{\sqrt{\lambda_1^2-\lambda^2}}\right) \Pi (\alpha^2,k) +\frac{\lambda_2^3}{\sqrt{\lambda_2^2-\lambda^2}} K (k) \\
&\qquad \qquad \geq \bigg[ 
\frac{\lambda_2 (\lambda_2^2-\lambda_1^2)(\lambda_1^2-1)}{(\lambda_1^2+\lambda_2^2-1) \sqrt{\lambda_2^2-\lambda^2}}-\frac{\lambda_1 (b^2-\lambda_1^2)}{\sqrt{\lambda_1^2-\lambda^2}}
\bigg] \Pi (\alpha^2,k),  
\end{align*}
and the right-hand side is non-negative if and only if 
\[
\lambda_2^2 (\lambda_2^2-\lambda_1^2)^2 (\lambda_1^2-1)^2 (\lambda_1^2-\lambda^2)
-\lambda_1^2(\lambda_1^2+\lambda_2^2-1)^2 (\lambda_2^2-\lambda^2)(b^2-\lambda_1^2)^2\geq 0. 
\]
In order to reduce this inequality to a polynomial of a single valuable $\lambda_1$, 
we now choose $\lambda=1/2$, $\lambda_2=11/10=1.1$. In this case, 
$\alpha/2-\lambda_1^2<0$ for $1\leq \lambda_1\leq \lambda_2=11/10$. 

Note that 
\begin{align*}
(b^2-\lambda_1^2)^2 
&= \left(\frac{\alpha}{2}-\lambda_1^2+\frac{\sqrt{\alpha^2-4\lambda\lambda_1\lambda_2}}{2}\right)^2 \\
&= \left(\frac{\alpha}{2}-\lambda_1^2\right)^2 +\frac{\alpha^2-4\lambda\lambda_1\lambda_2}{4}
+\underbrace{\left(\frac{\alpha}{2}-\lambda_1^2\right)}_{< 0}\sqrt{\alpha^2-4\lambda\lambda_1\lambda_2}. 
\end{align*}
Also, we may estimate, for $1\leq \lambda_1\leq \lambda_2=11/10$, 
\[
\lambda_2^2 (\lambda_2^2-\lambda_1^2)^2 (\lambda_1^2-1)^2 (\lambda_1^2-\lambda^2) 
-\lambda_1^2(\lambda_1^2+\lambda_2^2-1)^2 (\lambda_2^2-\lambda^2)
\bigg[ \left(\lambda_1^2-\frac{\alpha}{2}\right)^2 +\frac{\alpha^2-4\lambda\lambda_1\lambda_2}{4} \bigg]<0. 
\]
Thus, the above inequality can be reduced to 
\begin{align*}
& -\bigg[\lambda_2^2 (\lambda_2^2-\lambda_1^2)^2 (\lambda_1^2-1)^2 (\lambda_1^2-\lambda^2) 
-\lambda_1^2(\lambda_1^2+\lambda_2^2-1)^2 (\lambda_2^2-\lambda^2)
\bigg[ \left(\lambda_1^2-\frac{\alpha}{2}\right)^2 +\frac{\alpha^2-4\lambda\lambda_1\lambda_2}{4} \bigg]\bigg]^2 \\
& \hspace{170pt} +\bigg[\lambda_1^2(\lambda_1^2+\lambda_2^2-1)^2 (\lambda_2^2-\lambda^2)
\left(\lambda_1^2-\frac{\alpha}{2}\right)\bigg]^2 (\alpha^2-4\lambda\lambda_1\lambda_2)\geq 0. 
\end{align*}
The left-hand side of the above is a polynomial of $\lambda_1$. 
In fact, it is equal to a constant multiple of $(\lambda_2-\lambda_1)^2(\lambda_1^2-1)(\lambda_1+\lambda) \varphi (\lambda_1)$, 
where 
$\varphi (\lambda_1)$ is a polynomial of $\lambda_1$ 
which is positive for $1\leq \lambda_1\leq \lambda_2=11/10$. 
Since a polynomial of $\lambda$, $\lambda_1$, $\lambda_2$ is continuous, we obtain the proof. 

\subsection{Estimate for $\xi_2$ from above}

Straightforward calculations imply 
\[
\xi_2 
= \frac{1}{2}
\int_{\lambda_1^2}^{\lambda_2^2} \frac{s}{\sqrt{s(s-\lambda^2)(s-1)(s-\lambda_1^2)(\lambda_2^2-s)}}ds 
-\frac{b^2}{2} \int_{\lambda_1^2}^{\lambda_2^2} \frac{ds}{\sqrt{s(s-\lambda^2)(s-1)(s-\lambda_1^2)(\lambda_2^2-s)}}
\]
and  
\begin{align}
 \int_{\lambda_1^2}^{\lambda_2^2} \frac{s}{\sqrt{s(s-\lambda^2)(s-1)(s-\lambda_1^2)(\lambda_2^2-s)}}ds 
&\leq 
 \frac{\sqrt{\lambda_2^2-\lambda^2}}{\lambda_2}
\int_{\lambda_1^2}^{\lambda_2^2} \frac{s}{(s-\lambda^2)\sqrt{(s-1)(s-\lambda_1^2)(\lambda_2^2-s)}}ds, 
\label{period_estimate_1} \\
\int_{\lambda_1^2}^{\lambda_2^2} \frac{d s}{\sqrt{s(s-\lambda^2)(s-1)(s-\lambda_1^2)(\lambda_2^2-s)}} 
& \geq 
 \frac{\sqrt{\lambda_1^2-\lambda^2}}{\lambda_1}
\int_{\lambda_1^2}^{\lambda_2^2} \frac{d s}{(s-\lambda^2)\sqrt{(s-1)(s-\lambda_1^2)(\lambda_2^2-s)}} 
\label{period_estimate_2}. 
\end{align}
By \textbf{[235.18]} and \textbf{[340.01]}, 
the right-hand side of \eqref{period_estimate_1} becomes 
\begin{align*}
&\int_{\lambda_1^2}^{\lambda_2^2} \frac{s}{(s-\lambda^2)\sqrt{(s-1)(s-\lambda_1^2)(\lambda_2^2-s)}}ds \\
&\qquad =\frac{2}{(1-\lambda^2)(\lambda_1^2-\lambda^2)\sqrt{\lambda_2^2-1}}
[ (\lambda_1^2-\lambda^2) K(k) - \lambda^2 (\lambda_1^2-1) \Pi (\alpha^2,k)], 
\end{align*}
where 
\[
k^2 = \frac{\lambda_2^2-\lambda_1^2}{\lambda_2^2-1}, \qquad 
\alpha^2 = \frac{(\lambda_2^2-\lambda_1^2)(1-\lambda^2)}{(\lambda_2^2-1)(\lambda_1^2-\lambda^2)} < k^2, \qquad
\alpha_1^2 = \frac{\lambda_2^2-\lambda_1^2}{\lambda_1^2(\lambda_2^2-1)}. 
\]
By \textbf{[235.17]} 
and \textbf{[339.01]}, 
the right-hand side of \eqref{period_estimate_2} becomes 
\begin{align*}
&\int_{\lambda_1^2}^{\lambda_2^2} \frac{ds}{(s-\lambda^2)\sqrt{(s-1)(s-\lambda_1^2)(\lambda_2^2-s)}} \\
&\qquad =\frac{2}{(1-\lambda^2)(\lambda_1^2-\lambda^2)\sqrt{\lambda_2^2-1}}
[ (\lambda_1^2-\lambda^2) K(k) - (\lambda_1^2-1) \Pi (\alpha^2,k)], 
\end{align*}
where 
\[
k^2 = \frac{\lambda_2^2-\lambda_1^2}{\lambda_2^2-1}, \quad 
\alpha^2 = \frac{(\lambda_2^2-\lambda_1^2)(1-\lambda^2)}{(\lambda_2^2-1)(\lambda_1^2-\lambda^2)} < k^2. 
\]
Hence, we have 
\begin{align*}
\xi_2 &\leq \frac{1}{\lambda_1 \lambda_2 (1-\lambda^2)(\lambda_1^2-\lambda^2)\sqrt{\lambda_2^2-1}} 
 \bigg[ \lambda_1\sqrt{\lambda_2^2-\lambda^2} \bigg[(\lambda_1^2-\lambda^2) K(k) - \lambda^2 (\lambda_1^2-1) \Pi (\alpha^2,k) \bigg] \\
&\hspace{160pt} - b^2 \lambda_2\sqrt{\lambda_1^2-\lambda^2} \bigg[(\lambda_1^2-\lambda^2) K(k) - (\lambda_1^2-1) \Pi (\alpha^2,k) \bigg] \bigg]. 
\end{align*}
It corresponds to Case~{III} in \textbf{[400.01]}. By virtue of \textbf{[414.01]}, setting 
\[
\sin \beta = \frac{\alpha}{k} = \sqrt{\frac{1-\lambda^2}{\lambda_1^2-\lambda^2}}, 
\]
we have 
\begin{align}
\nonumber &\lambda_1\sqrt{\lambda_2^2-\lambda^2} \bigg[(\lambda_1^2-\lambda^2) K(k) - \lambda^2 (\lambda_1^2-1) \Pi (\alpha^2,k) \bigg] \\
\nonumber &\qquad \qquad \qquad - b^2 \lambda_2\sqrt{\lambda_1^2-\lambda^2} \bigg[(\lambda_1^2-\lambda^2) K(k) - (\lambda_1^2-1) \Pi (\alpha^2,k) \bigg] \\
\nonumber & = \lambda_2\bigg[\lambda_1\sqrt{1-\frac{\lambda^2}{\lambda_2^2}} \bigg[\lambda_1^2 (1-\lambda^2)  -\lambda^2\sqrt{ \frac{(1-\lambda^2)(\lambda_1^2-\lambda^2)(\lambda_2^2-1)}{\lambda_2^2-\lambda^2}} Z(\beta,k) \bigg] \\
& \qquad -\frac{b^2}{\lambda_2}\sqrt{\lambda_1^2-\lambda^2} \lambda_2 \bigg[1-\lambda^2  -\sqrt{ \frac{(1-\lambda^2)(\lambda_1^2-\lambda^2)(\lambda_2^2-1)}{\lambda_2^2-\lambda^2}} Z(\beta,k) \bigg] \bigg] K(k), 
\label{Jacobi_Zeta_estimate}
\end{align}
where $Z(\beta,k)$ is the Jacobian Zeta function. 

\subsubsection{Proof of Lemma~\ref{lm:6}} 

For an arbitrary fixed $\lambda_1>1$, we consider the case $\lambda_2\to \infty$, where $\lambda_2$ depends on $\lambda_1$. In this case, $k\to 1$ and $\beta$ is invariant. 

It follows from \textbf{[111.04]}, \textbf{[140.01]} that $Z (\beta,k) \to Z(\beta,1)=E(\beta,1)=\sin\beta$ holds. 
Note that 
\[
\lambda_1\sqrt{1-\frac{\lambda^2}{\lambda_2^2}}\bigg[\lambda_1^2 (1-\lambda^2)  -\lambda^2\sqrt{ \frac{(1-\lambda^2)(\lambda_1^2-\lambda^2)(\lambda_2^2-1)}{\lambda_2^2-\lambda^2}} Z(\beta,k)\bigg] \to \lambda_1(1-\lambda^2) (\lambda_1^2-\lambda^2)
\]
and
\[
\frac{b^2}{\lambda_2}\sqrt{\lambda_1^2-\lambda^2}\to (\lambda_1+\lambda-1)\sqrt{\lambda_1^2-\lambda^2}.
\]
On the other hand, since 
\[
1-\lambda^2  -\sqrt{ \frac{(1-\lambda^2)(\lambda_1^2-\lambda^2)(\lambda_2^2-1)}{\lambda_2^2-\lambda^2}} Z(\beta,k) \to 0, 
\]
the product
\[
\lambda_2 \bigg[1-\lambda^2  -\sqrt{ \frac{(1-\lambda^2)(\lambda_1^2-\lambda^2)(\lambda_2^2-1)}{\lambda_2^2-\lambda^2}} Z(\beta,k) \bigg]
\]
is an indeterminant product.  Thus, we need to check the order of this product as $\lambda_2\to \infty$. 

By virtue of \textbf{[900.05]}, \textbf{[900.10]}, \textbf{[902.01]}, \textbf{[903.01]}, we have 
\begin{align*}
Z(\beta,k) 
& = \frac{1}{K (k)} [ K (k) E (\beta,k) - E (k) F(\beta,k)] \\
& = \frac{1}{\log \frac{4}{\sqrt{1-k^2}}+\frac{1}{4}[ \log \frac{4}{\sqrt{1-k^2}}-1 ](1-k^2)+\cdots} 
\bigg( \sin\beta\log \frac{4}{\sqrt{1-k^2}}-\log \frac{1+\sin\beta}{\cos\beta} \\
& \hspace{85pt} +\bigg[ -\frac{1}{4} \sin\beta\log \frac{4}{\sqrt{1-k^2}} 
   -\frac{1}{4} \sin\beta+\frac{\sin\beta}{4\cos^2\beta} \bigg] (1-k^2) +\cdots\bigg). 
\end{align*}
It follows from 
\[
1-k^2 = \frac{\lambda_1^2-1}{\lambda_2^2-1}
\]
and 
\[
\left(\sqrt{\lambda_1^2-\lambda^2}+\sqrt{1-\lambda^2}\right)^2-\sqrt{\lambda_1^2-1}^2
=2(1-\lambda^2) +2\sqrt{(1-\lambda^2)(\lambda_1^2-\lambda^2)} >0 
\]
that 
\[
\frac{\lambda_2 \log \frac{1+\sin\beta}{\cos\beta}}{\log \frac{4}{\sqrt{1-k^2}}} 
= \frac{\lambda_2 \log \frac{\sqrt{\lambda_1^2-\lambda^2}+\sqrt{1-\lambda^2}}{\sqrt{\lambda_1^2-1}}}{\log \frac{4}{\sqrt{1-k^2}}} \to\infty. 
\] 
Hence, as $\lambda_2\to \infty$, the important part of 
\begin{align*}
\lambda_2 \bigg[ 1-\lambda^2  -\sqrt{ \frac{(1-\lambda^2)(\lambda_1^2-\lambda^2)(\lambda_2^2-1)}{\lambda_2^2-\lambda^2}} Z(\beta,k)  \bigg]
\end{align*}
is 
\begin{align*}
& \lambda_2 \bigg[ 1-\lambda^2 -(1-\lambda^2)\sqrt{ \frac{\lambda_2^2-1}{\lambda_2^2-\lambda^2}}
+\frac{\log \frac{\sqrt{\lambda_1^2-\lambda^2}+\sqrt{1-\lambda^2}}{\sqrt{\lambda_1^2-1}}}{\log \frac{4}{\sqrt{1-k^2}}} 
\sqrt{ \frac{(1-\lambda^2)(\lambda_1^2-\lambda^2)(\lambda_2^2-1)}{\lambda_2^2-\lambda^2}}\bigg] \\
&=\lambda_2 \bigg[ \frac{(1-\lambda^2)^2}{(\lambda_2^2-\lambda^2) \bigg( 1+\sqrt{ \frac{\lambda_2^2-1}{\lambda_2^2-\lambda^2}} \bigg)} 
+\frac{\log \frac{\sqrt{\lambda_1^2-\lambda^2}+\sqrt{1-\lambda^2}}{\sqrt{\lambda_1^2-1}}}{\log \frac{4}{\sqrt{1-k^2}}} 
\sqrt{ \frac{(1-\lambda^2)(\lambda_1^2-\lambda^2)(\lambda_2^2-1)}{\lambda_2^2-\lambda^2}}\bigg] 
\to \infty.
\end{align*}
Therefore, $\eqref{Jacobi_Zeta_estimate}\to -\infty$, and the proof is completed. 

\section*{Acknowledgements}
The authors thank Professors Wayne Rossman, Kenji Takeda, and Martin Traizet for fruitful discussions. 

%
\bibliographystyle{amsplain}

\end{document}